\documentclass[11pt,a4paper,leqno]{amsart}
\usepackage{amssymb}
\numberwithin{equation}{section}

\newtheoremstyle{theorem}{3pt}{3pt}%
{\it}
{}
{\bfseries}
{:}
{.5em}
{}
\theoremstyle{theorem}
\newtheorem{theorem}{Theorem}[section]

\newtheorem{proposition}[theorem]{Proposition}

\newtheorem{definition}[theorem]{Definition}

\newtheoremstyle{example}{3pt}{3pt}%
{}
{}
{\sc}
{:}
{.5em}
{}
\theoremstyle{example}

\newtheoremstyle{remark}{3pt}{3pt}%
{}
{}
{\sc}
{:}
{.5em}
{}
\theoremstyle{remark}

\numberwithin{equation}{section}

\newcommand{\hl}{\\\hline}

\newcommand{\thismonth}{\ifcase\month\or
  January\or February\or March\or April\or May\or June\or
  July\or August\or September\or October\or November\or December\fi
  \space\number\year}
\makeatletter
\newcommand{\low}{\@ifnextchar^{}{^{\vphantom x}}}
\newcommand{\high}{\@ifnextchar_{}{_{\vphantom I}}}
\makeatother
\DeclareSymbolFont{script}{U}{eus}{m}{n}
\DeclareSymbolFontAlphabet{\mathscr}{script}
\DeclareMathSymbol{\EuWedge}{0}{script}{"5E}
\DeclareMathAlphabet{\mathrmsl}{OT1}{cmr}{m}{sl}
\newcommand{\rssymb}[2]{\newcommand{#1}{{\mathrmsl{#2}}}}
\newcommand{\calsymb}[2]{\newcommand{#1}{{\mathcal{#2}}}}
\newcommand{\bbsymb}[2]{\newcommand{#1}{{\mathbb{#2}}}}
\newcommand{\lieoper}[2]{\newcommand{#1}{\mathop
  {\mathfrak{#2}\null}\nolimits}}
\newcommand{\oper}[3][n]{\newcommand{#2}{\mathop
  {\mathrm{#3}\null}\ifx n#1\nolimits\else\limits\fi}}
\newcommand{\rsoper}[3][n]{\newcommand{#2}{\mathop
  {\mathrmsl{#3}\null}\ifx n#1\nolimits\else\limits\fi}}
\bbsymb\C{C} \bbsymb\F{F} \bbsymb\HQ{H}\bbsymb\N{N} \bbsymb\Q{Q}
\bbsymb\R{R} \bbsymb\U{U} \bbsymb\V{V} \bbsymb\W{W} \bbsymb\Z{Z}
\bbsymb\bbf{F} \bbsymb\bbk{K} \bbsymb\bbi{I} \bbsymb\bbl{L} \bbsymb\bbo{O}
\bbsymb\bbj{J}
\bbsymb\bby{Y}
\bbsymb\bbp{P}
\bbsymb\bba{A}
\calsymb\cA{A} \calsymb\cB{B} \calsymb\cC{C} \calsymb\cD{D} \calsymb\cE{E}
\calsymb\cF{F} \calsymb\cG{G} \calsymb\cH{H} \calsymb\cI{I} \calsymb\cJ{J}
\calsymb\cK{K} \calsymb\cL{L} \calsymb\cM{M} \calsymb\cN{N} \calsymb\cO{O}
\calsymb\cP{P} \calsymb\cQ{Q} \calsymb\cR{R} \calsymb\cS{S} \calsymb\cT{T}
\calsymb\cU{U} \calsymb\cV{V} \calsymb\cW{W} \calsymb\cX{X} \calsymb\cY{Y}
\calsymb\cZ{Z}
\renewcommand{\geq}{\geqslant} 
\oper\End{End} \oper\Hom{Hom}                    
\oper\Sym{Sym} \oper\Skew{Skew}
\oper\Aut{Aut}                                   
\oper\GL{GL} \oper\SL{SL}\oper\Symp{Sp}
\oper\CO{CO} \oper\On{O} \oper\SO{SO} \oper\Pin{Pin} \oper\Spin{Spin}
\oper\CU{CU} \oper\Un{U} \oper\SU{SU} \oper\PSU{PSU}
\rsoper\Diff{Diff} \rsoper\SDiff{SDiff}
\lieoper\der{der}                                
\lieoper\gl{gl} \lieoper\sgl{sl}\lieoper\symp{sp}
\lieoper\co{co} \lieoper\so{so} \lieoper\spin{spin}
\lieoper\cu{cu} \lieoper\un{u}  \lieoper\su{su}
\rsoper\Vect{Vect} \rsoper\Ham{Ham}
\def\la#1{\hbox to #1pc{\leftarrowfill}}
\def\ra#1{\hbox to #1pc{\rightarrowfill}}

\newcommand{\norm}[2][]{|\mkern-2mu|#2|\mkern-2mu|
  _{\lower1pt\hbox{${}_{#1}$}}}
\newcommand{\Norm}[2][]{\bigl|\mkern-3mu\bigr|#2\bigr|\mkern-3mu\bigr|
  _{\lower1pt\hbox{${}_{#1}$}}}

\rsoper\dimn{dim}                           
\rsoper\grad{grad}                          
\rsoper\kernel{ker}\rsoper\image{im}        
\rsoper\alt{alt}   \rsoper\sym{sym}         
\rsoper\Ad{Ad}     \rsoper\ad{ad}           
\rsoper\CoAd{CoAd} \rsoper\coad{coad}       
\rsoper\trace{tr}  \rsoper\trfree{tf}       
\rsoper\detm{det}                           
\rsoper\Vol{Vol}                            
\rsoper\divg{div}                           
\rsoper\sign{sign}                          
\rssymb\iden{id}                            
\rssymb\vol{vol}                            
\oper\Imag{Im}\oper\Real{Re}                
\newcommand{\sd}{{\raise1pt\hbox{$\scriptscriptstyle +$}}}
\newcommand{\asd}{{\raise1pt\hbox{$\scriptscriptstyle -$}}}
\newcommand{\sdasd}{{\raise1pt\hbox{$\scriptscriptstyle\pm$}}}
\newcommand{\asdsd}{{\raise1pt\hbox{$\scriptscriptstyle\mp$}}}

\rsoper\scal{scal}
\def\kahl/{k\"ahler}
\def\Kahl/{K{\"a}hler}
\begin{document}
\title[Smale Manifolds and Lorentzian Sasaki-Einstein Geometry]
{A Note On Smale Manifolds and Lorentzian Sasaki-Einstein Geometry}
\author[R.Gomez]{Ralph R. Gomez}
\address{Ralph R. Gomez: Swarthmore College}
\email{rgomez1@swarthmore.edu}

\date{\thismonth}
\begin{abstract}
In this note, we construct new examples of Lorentzian Sasaki-Einstein (LSE) metrics on Smale manifolds $M.$
It has already been established in \cite{Gmz2} that such metrics exist on the so-called torsion
free Smale manifolds, i.e. the $k$-fold connected sum of $S^{2}\times S^{3}.$ Now, we show that
LSE metrics exist on Smale manifolds for which $H_{2}(M,\mathbb{Z})_{tor}$ is nontrivial.
In particular, we show that most simply-connected positive Sasakian rational homology $5$-spheres are
also negative Sasakian (hence Lorentzian Sasaki-Einstein).
Moreover, we show that for each pair of positive integers $(n,s)$ with $n,s >1$, there exists a Lorentzian Sasaki-Einstein Smale manifold
$M$ such that $H_{2}(M,{\mathbb{Z}})_{tors}=(\mathbb{Z}/n)^{2s}$. Finally, we are able to construct
so-called mixed Smale manifolds (connect sum of torsion free Smale manifolds with rational homology spheres) which admit LSE metrics and have arbitrary second Betti number. This gives infinitely many
examples which do not admit positive Sasakian structures. These results partially address the open problems
in \cite{BG6}.
\end{abstract}
\maketitle
\vspace{-2mm}

\bigskip

\section{Introduction}
Let $M$ be a compact, simply-connected, $5$-manifold with $w_{2}(M)=0$. By a theorem of Smale \cite{Sm}(see section 2 also),
$H_{2}(M,\mathbb{Z})$ uniquely determines $M.$ Following \cite{BN1}, let us call such manifolds $M$ \textit{Smale} manifolds
and divide them up into three classes: \noindent \emph{i.}) torsion free Smale manifolds with positive
second Betti number, that is $k$-fold connected sums of $S^{2}\times S^{3}$  for some $k$ \emph{ii.)} rational homology spheres, including $S^{5}$ and
\emph{iii.)} mixed Smale manifolds, which are the connected sums of torsion free Smale manifolds
with rational homology spheres. There has been a flurry of work (see \cite{BG6} for sweeping survey) dedicated to determining
which Smale manifolds admit so-called quasi-regular positive Sasakian structures. A quasi-regular positive (negative)
Sasakian manifold can be viewed as a principal circle orbibundle over a complex algebraic
orbifold $(X,\Delta=\sum(1-\frac{1}{m_{i}})D_{i})$ such that the orbifold first Chern class is
positive (negative). In \cite{BGN4}, it was shown that positive Sasakian structures exist on an arbitrary connected
sum of torsion free Smale manifolds. Then in \cite{Kol1} Koll\'ar established a classification theorem
for simply-connected rational homology $5$-spheres admitting positive Sasakian structures. There are partial results
for the mixed Smale manifolds. See \cite{BG6} for more complete survey and \cite{BN1} for more recent results. \\
\indent Negative Sasakian geometry on Smale manifolds is not nearly as developed. In \cite{Gmz2}, it was shown
that negative Sasakian structures exist on the $k$-fold connected sum of $S^{2}\times S^{3}$ for all $k.$
It was shown in \cite{Gmz1} that the simply-connected rational homology $5$-sphere $M_{k}$ admits both positive and negative Sasakian structures for $k\geq 5$. This example was also given later in the book \cite{BG6}. Until now, negative Sasakian
structures on other simply-connected positive Sasakian rational homology $5$-spheres have not been constructed. It turns out that a negative Sasakian structure can give rise to negative Sasaki $\eta$-Einstein as well as Lorentzian Sasaki-Einstein metrics.
Thus, negative Sasakian geometry is a strong tool in constructing such manifolds with these metric properties. In fact, Boyer and Galicki
formulated the following open problems on negative Sasakian geometry in dimension five on pages $359-360$ in \cite{BG6}:\\

\noindent \emph{a.) Determine which simply-connected rational homology 5-spheres admit negative Sasakian structures.}\\
\noindent \emph{b.) Determine which torsion groups correspond to Smale-Barden manifolds admitting negative Sasakian structures.}\\

\noindent The more general Smale-Barden five-manifolds correspond to the non-spin case and that will not be addressed in this paper.
The principle aim of this note is to make a significant first step towards solutions to the above problems. The main theorem is:\\
\textbf{Theorem}\\
\noindent (1.) \textit{The following simply-connected positive Sasakian rational homology $5$-spheres admit
a negative Sasakian structure:
$$M_{m},\hspace{.3cm} 2M_{3}, \hspace{.3cm} 3M_{3},\hspace{.3cm} 4M_{3},\hspace{.3cm}2M_{4},\hspace{.3cm} 2M_{5}$$
where $m$ is a positive integer with $m\geq 5$ and $m\neq 30j$ for some positive integer $j$. Consequently, these manifolds admit both positive and negative Sasakian structures.\\}


\noindent (2.) \textit{For every pair of positive integers $(n,s)$ with $n,s > 1$  there exists a Smale manifold $M$ which admits a negative Sasakian structure such that
$H_{2}(M,\mathbb{Z})_{tors}=(\mathbb{Z}/n)^{2s}$. }\\

\noindent (3.) \textit{There exists mixed Smale manifolds $M$ which are negative Sasakian and have arbitrary second Betti number
 but do not admit any positive Sasakian structure. Moreover, all the manifolds listed in 1.) - 3.) can be realized as a links of an isolated hypersurface singularity at the origin and admit negative Saskaki $\eta$-Einstein, hence Lorentzian Sasaki-Einstein metrics.}\\

The only manifolds from the classification list of positive Sasakian
simply connected rational homology $5$-spheres missing
in the above theorem is $nM_{2}$ and $M_{m}$ where $m=2,3,4.$ It is interesting to note that generally, it is not known
for which $n$ the manifolds $nM_{2}$ are Sasaki-Einstein of positive scalar curvature (See \cite{BN1} for some recent known values of $n.$)\\
\indent It is worthwhile to contrast the second part of the above theorem with $5$-dimensional simply-connected positive Sasakian geometry. Let $M$
be a positive quasi-regular Sasakian manifold which is simply connected. Then the torsion group in $H_{2}(M,\mathbb{Z})$ is constrained to be one of the following \cite{Kol1}:
\begin{equation}
(\mathbb{Z}/m)^2,(\mathbb{Z}/5)^4,(\mathbb{Z}/4)^4,(\mathbb{Z}/3)^4,(\mathbb{Z}/3)^6,(\mathbb{Z}/3)^8,(\mathbb{Z}/2)^{2n}
\end{equation}
for any $n,m\in \mathbb{Z}^{+}.$\\
\indent Continuing the comparison with positive Sasaki-Einstein geometry of Smale manifolds, there are no
examples of Sasaki-Einstein structures on mixed Smale manifolds with second Betti number larger than nine. The
third part of our theorem illustrates that is far from the case in Lorentzian Sasaki-Einstein geometry.\\
\indent The organization of the note is as follows. The second section recalls the Sasakian geometry on $5$-dimensional links.
The third section gives the construction. Finally, the fourth section gives some comments and remarks.

\section{Basics of Negative Sasakian Geometry On Links of Isolated Hypersurface Singularities}

A \emph{Sasakian} manifold is a smooth Riemannian manifold $M^{2n+1}$ with some structure tensors $\mathcal{S}=(\xi,\eta,\Phi,g)$ that makes $M^{2n+1}$ into a \emph{normal contact metric structure}. (We may sometimes refer to $\mathcal{S}$ as a Sasasakian structure on the manifold.) The Reeb vector field $\xi$ is a Killing vector field.  The Reeb vector field foliates $M$ by one-dimensional leaves. The $1$-form $\eta$ is a contact form and
$\Phi$ is an endomorphism of the tangent bundle such that $\Phi$ restricted to ker$\eta=\mathcal{D}$ is an integrable almost complex
structure. Moreover, if the leaves are compact, we say the Sasakian manifold
is \emph{quasi-regular}. The corresponding leaf space $(X,\Delta)$ is a K\"ahler orbifold and so one can view
 a quasi-regular Sasakian manifold as the total space of principle circle orbibundle (or Seifert $S^1$ bundle). Another equally useful definition of a Sasaki manifold involves looking at the cone geometry. Denote the metric cone on $M$ by $(C(M),\overline{g})=(M \times\mathbb{R}_{+},dr^{2}+r^{2}g)$, where $r>0$. The quick definition of a Sasakain manifold $M$ is: the cone $C(M)$ is K\"ahler iff $M$ is Sasakian. Furthermore, a Sasaki-Einstein manifold with Einstein constant $2n$ corresponds to $C(M)$ being Ricci-flat Calabi-Yau.\\
\indent A plethora of examples of quasi-regular Sasakian manifolds can be obtained by constructing links of isolated hypersurface singularties.
This construction was initially implemented by Boyer and Galicki in \cite{BG2} to construct numerous examples of
Sasaki-Einstein manifolds. Recall, the weighted $\mathbb{C}^{*}$ action on $\mathbb{C}^{n+1}$ defined by
$$(z_{0},...,z_{n})\longmapsto (\lambda^{w_{0}}z_{0},...,\lambda^{w_{n}}z_{n})$$
where $w_{i}$ are the weights which are positive integers and $\lambda \in \mathbb{C}^{*}$. We use the standard notation $\textbf{w}=(w_0,...,w_n)$ to denote a weight vector. In addition, we assume
\begin{equation}
gcd(w_{0},...,w_{n})=1.
\end{equation}
\begin{definition}
A polynomial $f\in \mathbb{C}[z_{0},...,z_{n}]$ is weighted homogenous if it
satisfies $$f(\lambda^{w_{0}}z_{0},...,\lambda^{w_{n}}z_{n})=\lambda^{d}f(z_{0},...,z_{n})$$
for any $\lambda \in \mathbb{C}^{*}$ and the positive integer $d$ is the degree of $f$.
\end{definition}
Let the \emph{weighted affine cone} $C_{f}$ be a hypersurface in $\mathbb{C}^{n+1}$
defined by the equation $f(z_{0},...,z_{n})=0$ where the origin is an isolated singularity.
The link of the isolated hypersurface singularity, denoted by $L_{f}$, is defined as $$L_{f}=C_{f}\cap S^{2n+1}$$ where
$S^{2n+1}$ is viewed as the unit sphere in $\mathbb{C}^{n+1}$. Milnor showed in \cite{Milnor} that $L_{f}$ is a compact, smooth $2n-1$ manifold
which is $n-2$ connected. Furthermore, it is widely known that $L_f$ admits a quasi-regular Sasakian structure
inherited from the ``weighted Sasakian structure'' $\mathcal{S}_{\textbf{w}}=(\xi_{\textbf{w}},\eta_{\textbf{w}},\Phi_{\textbf{w}},g_{\textbf{w}})$ on the weighted sphere $S_{\textbf{w}}^{2n+1}$. The flow of the Reeb vector field $\xi_{\textbf{w}}$ generates a locally free weighted $S^{1}_{\textbf{w}}$ action such that the quotient of
$S_{\textbf{w}}^{2n+1}$ by this circle action produces a K\"ahler orbifold $\mathbb{P}(\textbf{w})$. Moreover the quotient space $\mathcal{Z}_{f}$ of $L_{f}$ by $S^{1}_{\textbf{w}}$ is also a K\"ahler orbifold which is a hypersurface in weighted projective space. (See \cite{BG6} Chapter 9 for details.)
The Boyer-Galicki construction is typically described succinctly by the commutative diagram:
\begin{equation}
\begin{matrix}
L_{f} &&\longrightarrow&& S^{2n+1}_{\textbf{w}} \\
 \Big\downarrow \pi &&&&\Big\downarrow\\
\mathcal{Z}_{f} &&\longrightarrow&& \mathbb{P}(\textbf{w}),
\end{matrix}
\end{equation}
where the vertical arrows are Seifert-$S^1$ bundles, the horizontal rows are Sasakian and K\"ahler embeddings respectively \cite{BGN2} and
$\pi^{*}c_{1}^{orb}(\mathcal{Z}_f)=c^{B}_{1}(\mathcal{F}_{\xi})$, where $c_{1}^{orb}(\mathcal{Z}_f)$ is the orbifold first Chern class and
$c^{B}_{1}(\mathcal{F}_{\xi})$ is the basic first Chern class of the characteristic foliation $\mathcal{F}_{\xi}$ generated by the quasiregular Reeb vector field $\xi.$\\
\indent We will be studying Sasakian structures on a particular class of $5$-manifolds. The main reason is because of the following classification
theorem of Smale \cite{Sm}.
\begin{theorem}\emph{(Smale)} Let $M$ be a closed simply-connected $5$-manifold with $w_{2}(M)=0$. Then the finitely generated Abelian group $H_{2}(M,\mathbb{Z})$ uniquely determines $M$ and is diffeomorphic to $$kM_{\infty}\#M_{m_{1}}\#\cdots M_{m_{n}}$$
where $m_{i}|m_{i+1}$ and $m_{1}>1$ if $M$ is not $S^{5}.$
\end{theorem}
\noindent The $5$-manifolds $M_{m}$ are such that $H_{2}(M_{m},\mathbb{Z})=\mathbb{Z}/m\oplus \mathbb{Z}/m$ and $kM_{\infty}$
is the $k$-fold connected sum of $S^{2}\times S^{3}=M_{\infty}$, where $k=rankH_{2}(M,\mathbb{Z})$. Since we want to explore definite Sasakian structures on these manifolds, we need a
\begin{definition}
A Sasakian structure on $M$ is \emph{negative} \emph{(positive)} if the basic first Chern class
$c^{B}_{1}(\mathcal{F}_{\xi})$ is represented by a negative(positive) definite $(1,1)$-form.
\end{definition}
Therefore, obtaining a negative Sasakian structures on a link of an isolated hypersurface singularity corresponds to finding
a hypersurface in weighted projective space with negative orbifold first Chern class. This amounts to
\begin{proposition}
A link $L_{f}$ of degree $d$ is a negative Sasakian manifold if $d-\sum_{i}w_{i} > 0.$
\end{proposition}
\begin{proof}See, for example, \cite{BG6} Proposition $9.2.4$, for a proof.
\end{proof}
Indeed, part of the power of using links of isolated hypersurface singularities stems from the fact that the differential geometry and topology
of the link is encoded in the weights and degree of the weighted homogenous polynomial.\\
\indent We shall make use of the following result which was established in \cite{Kol1} which says that the presence of torsion in $H_{2}(M,\mathbb{Z})$
comes directly from the presence of nonrational curves in the branch divisor. More precisely,
\begin{theorem}\emph{(Koll\'ar)} Let $f:M\rightarrow (X,\Delta=\sum_{i}(1-\frac{1}{m_{i}})D_{i})$ be a smooth
Seifert bundle, where $M$ is a compact simply connected $5$-manifold, $\Delta$
is a branch divisor. Then
$$H_{2}(M,\mathbb{Z})=\mathbb{Z}^{k}\oplus\displaystyle \sum_{i}(\mathbb{Z}/m_{i})^{2g(D_{i})}.$$
\end{theorem}
\noindent In the theorem, $g(D_{i})$ denotes the genus of $D_{i}$, $m_{i}$ is the ramification index of $D_{i}$ and
$k=b_{2}(M)$.

\indent Perhaps a more useful aspect of negative Sasakian geometry is that such structures can give rise
to negative $\eta$-Einstein structures as well as Lorentzian Sasaki-Einstein structures.
Recall a smooth Riemannian manifold $M^{2n+1}$ is Sasaki $\eta$-Einstein if it is Sasakian and if the Riemannian metric $g$ on $M^{2n+1}$ satisfies a generalized Einstein condition $$Ric_g=\lambda g + \nu \eta \otimes \eta,$$
where $\lambda + \nu = 2n$ and $\eta$ is a contact $1$-form (assume dim $M >3$). Furthermore,
Lorentzian Sasaki-Einstein metrics on manifolds can be defined in a way similar to the positive definite
Sasaki-Einstein definition. We recall the following \cite{BGM}:
\begin{definition}
Let $(M^{2n+1},g)$ be a Lorentzian manifold of signature $(1,2n)$ where the
nowhere vanishing vector field $\xi$ is a time-like Killing vector field such that
$g(\xi,\xi)=-1$. We say $M$ is Sasakian if the cone $(C(M),\overline{g})=(M \times\mathbb{R}_{+},-dt^{2}+r^{2}g)$ is pseudo-K\"ahler and
Lorentzian Sasaki-Einstein if the cone is pseudo-Calabi-Yau with Einstein constant $-2n$.
\end{definition}

So, the relationship between negative Sasakian geometry on links, $\eta$-Einstein, and Lorentzian Sasaki-Einstein structures lies in the following:
\begin{theorem}
Let $L_f$ be a link of an isolated hypersurface singularity at the origin which is negative Sasakian. Then $M$
admits a negative Sasaki $\eta$-Einstein metric and Lorentzian Sasaki-Einstein metric.
\end{theorem}
\begin{proof}
This result follows from Theorem $17$ and Proposition $23$ of \cite{BGM}.
\end{proof}

\noindent It should be noted, therefore, that all of our examples of Lorentzian Sasaki-Einstein manifolds equivalently give noncompact Ricci flat pseudo Calabi-Yau manifolds. We are now ready to prove the main theorem.

\section{Proof of Main Theorem}

To construct negative Sasakian structures on the positive Sasakian rational homology $5$-spheres, we employ
the links of isolated hypersurface singularity approach implemented by Boyer and Galicki \cite{BG2}.
This entails finding the appropriate orbifolds in weighted projective space, subjected to the negativity condition, and possessing certain branch divisors of the appropriate genus. Then we will use a proposition in \cite{BG3} that allows one to state that the links
constructed are rational homology spheres. Using Theorems 2.2 and 2.5 allows us to determine the diffeomorphism type. Consider the following hypersurface in weighted projective space
$$\mathcal{Z}_{f}\subset \mathbb{P}(2qp,2\alpha p, 2\alpha q, \alpha (p-1))$$
where the hypersurface $\mathcal{Z}_{f}$ is defined as the zero locus of the weighted homogenous polynomial of degree $d=2qp\alpha$ (assume $q$,$p$,$\alpha$ are all distinct)
$$f=z_{0}^{\alpha}+z_{1}^{q}+z_{2}^{p}+z_{2}z_{3}^{2q}.$$ (These polynomials are of Type II in the Yau-Yu list \cite{YauYu}.) Furthermore, let us assume $gcd(q, 2\alpha p(p-1))=1$ and $(p,\alpha)=1$. With these assumptions together with $(2.1)$
it follows that $\alpha$, $q$ are odd
and $p$ is even. Therefore, we have $gcd(\alpha, 2qp)=1$ and so the link of $L_{f}$ is a simply-connected rational homology $5$-sphere by Proposition 2.1 of \cite{BG5} (or see Theorem 9.3.17 in \cite{BG6}).
There are two curves in the branch divisor: $C_{0}$ of degree $2qp$ corresponding to $z_{0}=0$ and $C_{1}$ of degree $ qp\alpha$
corresponding
to $z_{3}=0$. More precisely, $C_{0}=z_{1}^{q}+z_{2}^{p}+z_{2}z_{3}^{2q}$ and $C_{1}=z_{0}^{\alpha}+z_{1}^{q}+z_{2}^{p}$ with ramification indices
$\alpha$ and $2$ respectively. Hence,
on the hypersurface $\mathcal{Z}_{f}$, we have the branch divisor $\Delta=\frac{\alpha -1}{\alpha}C_{0}+\frac{1}{2}C_{1}.$
A negative Sasakian structure exists on $L_{f}$ as long as
\begin{equation}
2qp(\alpha-1)+\alpha(1-3p-2q)> 0.
\end{equation}
To determine the topology of $L_{f}$, we can apply Theorem 2.4. In particular, we must
compute the genus of the curves $C_{0}$ and $C_{1}$. Recall the genus formula (see for example \cite{Flet}) which we
shall use
\begin{equation}2g(C)=\frac{d^{2}}{w_{0}w_{1}w_{2}}-d\sum_{i<j}\frac{gcd(w_{i},w_{j})}{w_{i}w_{j}}+
\sum_{i}\frac{gcd(d,w_{i})}{w_{i}}-1.\end{equation}
where $C$ is a curve of degree in $d$ in the weighted projective space $\mathbb{P}(w_{0},w_{1},w_{2})$
By applying this formula (and using the gcd conditions defined above) to $C_{0}\subset \mathbb{P}(2p,2q,p-1)$ and $C_{1}\subset \mathbb{P}(qp,\alpha p, \alpha q)$, we find
$$2g(C_{0}) =\frac{(2qp)^2}{(2p)(p-1)(2q)}-2qp\left(\frac{2}{4qp}+\frac{1}{2p(p-1)}+\frac{1}{2q(p-1)}\right)+\left(\frac{1}{p-1}+1\right)=q-1.$$
Similarly, we can calculate the genus of $C_{1}\subset \mathbb{P}(qp,\alpha p, \alpha q)$ and here we find
$$2g(C_{1})=1-qp\alpha\left(\frac{p}{q\alpha p^2}+\frac{\alpha}{\alpha^{2}pq}+\frac{q}{q^{2}\alpha p}\right)+2=0.$$
This shows that $C_1$ is a rational curve which means that this curve will not contribute to torsion in $H_{2}(L_{f},\mathbb{Z})$, by Theorem 2.4.
So we have that $L_f$ is a negative Sasakian rational homology $5$ sphere and by Theorem 2.2, we have
$L_{f}$ is diffeomorphic to $\frac{q-1}{2}M_{\alpha}$ and hence $H_{2}(L_{f},\mathbb{Z})=(\mathbb{Z}/\alpha )^{q-1}$.
Note that $q$ is assumed to be odd.

By choosing the appropriate weights, we can obtain some of the positive
rational homology spheres in part $1$ of Theorem 1.1 which, due to our construction now admit negative Sasakian structures as well. We can get $2M_3$ by letting $q=5$, $\alpha = 3$ and $p=4$. Let $q=7$, $\alpha = 3$ and $p=4$, getting the rational homology sphere with torsion $(\mathbb{Z}/3)^6$. It should be noted that particular choices of $p$ will exhibit other negative Sasakian structures on a given rational homology sphere.
Again using the Type II polynomials as described in \cite{YauYu}, we can obtain the other rational homology spheres in the list. Since the computation is nearly identical to the above, the remaining cases are the first three entries in Table $1$ at the end of this section. The condition
that $m\neq 30j$ ensures that $M_{m}$ is positive Sasakian, as established in \cite{Kol1}.  \\
\indent For the second part of the theorem, we study the following
orbifold in weighted projective space cut out by the Brieskorn-Pham polynomial equation $f=z_{0}^{2}+z_{1}^{2l}+z_{2}^{l}+z_{3}^{2nl}=0$
of degree $d=2nl$ in the weighted projective space $\mathbb{P}(nl,n,2n,1)$. The negative Sasakian condition is
$n(l-3)>1$ so let $l \geq 4, n\geq 2.$ Note that we have a branch divisor for $z_3=0$ with ramification index $n$ so this generates the curve
$C=z_{0}^{2}+z_{1}^{2l}+z_{2}^{l}\subset\mathbb{P}(nl,n,2n)=\mathbb{P}(l,1,2)$ where this curve has degree $d=2l.$
As before, we compute the genus of this curve  $H_2(L_{f},\mathbb{Z})$ obtaining:$$2g(C_{f})=l-gcd(l,2).$$
Since $l$ is an integer greater than or equal to $4$, part two of the theorem is established.\\
\indent To establish the third part of the theorem, we need to construct a link
for which the second Betti number is arbitrary. To this end, we construct the orbifold hypersurface
$f=z_{0}^{k+1}+z_{1}^{k+1}+z_{2}^{k+1}+z_{0}z_{3}^{n}$ in the weighted projective space $\mathbb{P}(n,n,n,k)$ of degree $n(k+1).$ This was
investigated in an unpublished part of the author's thesis \cite{Gmz1}. Assume that $gcd(n,k)=1$. To calculate the second Betti number of the
link, we use the formula devised by Milnor and Orlik \cite{MilnorOrlick}. Recall that this involves computing
$$divisor\Delta(t)=\displaystyle\prod_{i}\big(\frac{\Lambda_{u_i}}{v_{i}}-1\big)=1+\sum a_{j}\Lambda_{j}$$
using the relation $\Lambda_{a}\Lambda_{b} = gcd(a,b)\Lambda_{lcm(a,b)}.$ The symbol $\frac{u_{i}}{v_{i}}$
is the irreducible representation of $\frac{d}{w_{i}}$, where $d$ is the degree of the polynomial and
$w_{i}$ are the weights. The formula for the second
Betti number of the link is $$b_{2}(L_{f})=1+\sum_{i}a_{i}.$$ Applying this procedure to
our example, we obtain \begin{equation}
divisor\Delta(t)=\big(\Lambda_{k+1}-1\big)^{3}\big(\frac{1}{k}\Lambda_{n(k+1)}-1\big)=k^{2}\Lambda_{n(k+1)}-(k^{2}-k+1)\Lambda_{k+1}+1\\
\end{equation}
and therefore $b_{2}(L_{f})=k$. The negative Sasakian condition is easily seen to be $n(k-2)>0.$\\
Of course, there is a branch divisor for $z_{3}=0$, with ramification index $n$, corresponding to $C=z_{0}^{k+1}+z_{1}^{k+1}+z_{2}^{k+1}\subset\mathbb{P}(1,1,1)=\mathbb{P}^{2}$ so using the usual genus formula for curves, we find $2g(C)=k(k-1)$. Then we compute $H_{2}(L_{f},\mathbb{Z})$ as above and we find, then,
that $$H_{2}(L_{f},\mathbb{Z})=\mathbb{Z}^{k}\oplus (\mathbb{Z}/n)^{k(k-1)}.$$ Hence, for each positive integer $k\geq 3$ we may
choose an $n$ such that the resulting torsion subgroup of $H_{2}(L_{f},\mathbb{Z})$ is not on the list in $(1.1)$\\
\indent  Now, since all of our examples are negative Sasakian links, by Theorem 2.7 they are
negative Sasaki $\eta$-Einstein as well as Lorentzian Sasaki-Einstein. The fact
that our examples are spin follows from Corollary 11.8.5 of \cite{BG6}.
This concludes the proof.\\

\begin{center}
\parbox{4.00in}{\small \textbf{Table 1}. Negative Sasakian rational homology $5$-spheres}
\vbox{\[\begin{array}{|c|c|c|c|} \hline \textbf{w}=(w_0,w_1,w_2,w_3) & {\rm
link\hspace{.1cm} L_{f}}
&degree& manifold \hl\hline
(15,12,4,28)\hspace{.1cm} & z_{0}^{4}+z_{1}^{5}+z_{2}^{15}+z_{2}z_{3}^2 & 60 & 2M_{4} \hl
(42,35,15,65)& z_{0}^{5}+z_{1}^{6}+z_{2}^{14}+z_{2}z_{3}^3  & 210 & 2M_{5}\hl
(68,51,6,33)& z_{0}^{3}+z_{1}^{4}+z_{2}^{34}+z_{2}z_{3}^{6} & 204 & 4M_{3}  \hl
(p,m,m((p+1)/4),m((p-1)/2)& z_{0}^{m}+z_{1}^{p}+z_{2}^{2}z_{3}+z_{3}^{2}z_{1} & mp & M_{m},m>4  \hl

\end{array}\]}
\vspace{3mm}
\end{center}

The entry in the very bottom row was used in \cite{BG3,BG5} to establish Sasaki metrics of positive Ricci curvature on certain simply-connected rational homology $5$-spheres and it was also realized in the author's thesis \cite{Gmz1} that one could use those links to obtain
the existence of negative Sasakian structures on certain rational homology spheres. An assumption is needed on
$m,p$ to make things work out and that is for each $m$ choose a prime $p$ of the form $p=4l-1$ such that $(m,p)=1$. The negativity condition
is easily seen to be $(m-4)(l-1)> 3$.

\section{Concluding Remarks}

\indent First, in positive Sasaki-Einstein geometry of Smale manifolds, an open problem is the following \cite{BG6}:
Suppose $k=b_{2}(M)>9.$ Then $M$ admits a Sasaki-Einstein structure if and only if $H_{2}(M,\mathbb{Z})_{tor}=0$
i.e. $M$ is diffeomorphic to a $k$-fold connected sum of $S^{2}\times S^{3}.$ Said differently,
mixed Smale Sasaki-Einstein $5$-manifolds with second Betti number bigger than $9$ do not exist.
By our main thorem, we see that such a statement is not true in the Lorentzian Sasaki-Einstein case.\\
\indent Secondly, because of the numerous examples of positive Sasakian manifolds that also admit
negative Sasakian structures, one is led to speculate that perhaps every positive Sasakian
manifold admits a negative Sasakian structure. The evidence seems to suggest that this is true for dimension
five. For the general case, it is not so clear.

\end{document}